\documentclass[11pt]{amsart}
\usepackage{amsmath,amssymb,latexsym,soul,cite,mathrsfs}

\usepackage{color,enumitem,graphicx}
\usepackage[colorlinks=true,urlcolor=blue,
citecolor=red,linkcolor=blue,linktocpage,pdfpagelabels,
bookmarksnumbered,bookmarksopen]{hyperref}
\usepackage[english]{babel}

\usepackage[left=2.9cm,right=2.9cm,top=2.8cm,bottom=2.8cm]{geometry}
\usepackage[hyperpageref]{backref}
\usepackage{ dsfont }
\usepackage[colorinlistoftodos]{todonotes}
\makeatletter
\providecommand\@dotsep{5}
\def\listtodoname{List of Todos}
\def\listoftodos{\@starttoc{tdo}\listtodoname}
\makeatother

\numberwithin{equation}{section}
\def\dis{\displaystyle}

\def\R {{\rm I}\hskip -0.85mm{\rm R}}
\def\N {{\rm I}\hskip -0.85mm{\rm N}}

\def\l{\lambda}
\def\m{\mu}

\def\O{\Omega}

\def\ov{\overline}

\def\dis{\displaystyle}

\newtheorem{theorem}{Theorem}[section]
\newtheorem{proposition}[theorem]{Proposition}
\newtheorem{lemma}[theorem]{Lemma}

\newtheorem{remark}{Remark}

\title[Study of a class of generalized Schr\"{o}dinger equations]
{Study of a class of generalized Schr\"{o}dinger equations}

\author[A. V. Santos]{Andrelino V. Santos}
\author[J. R. Santos Jr.]{Jo\~ao R. Santos J\'unior}
\author[A. Su\'arez]{Antonio Su\'arez}

\address[A. V. Santos]{\newline\indent Faculdade de Matem\'atica
\newline\indent 
Instituto de Ci\^{e}ncias Exatas e Naturais
\newline\indent 
Universidade Federal do Par\'a
\newline\indent
Avenida Augusto corr\^{e}a 01, 66075-110, Bel\'em, PA, Brazil}
\email{\href{mailto: andrellino77@gmail.com}{andrellino77@gmail.com}}

\address[J. R. Santos Jr.]{\newline\indent Faculdade de Matem\'atica
\newline\indent 
Instituto de Ci\^{e}ncias Exatas e Naturais
\newline\indent 
Universidade Federal do Par\'a
\newline\indent
Avenida Augusto corr\^{e}a 01, 66075-110, Bel\'em, PA, Brazil}
\email{\href{mailto: joaojunior@ufpa.br }{joaojunior@ufpa.br}}

\address[A. Su\'arez]{\newline\indent Fac. de Matem{\'a}ticas
\newline\indent 
Dpto. de Ecuaciones Diferenciales y An{\'a}lisis Num{\'e}rico
\newline\indent 
Univ. de Sevilla
\newline\indent
C/. Tarfia s/n, 41012, Sevilla, Spain.}
\email{\href{mailto: suarez@us.es}{suarez@us.es}}

\thanks{Jo\~ao R. Santos J\'unior was partially supported by CNPq-Proc. 302698/2015-9 and CAPES-Proc. 88881.120045/2016-01, Brazil. Antonio Su\'arez was partially supported by MTM2015-69875-P (MINECO/ FE\-DER, UE)}

\begin{document}

\maketitle
\begin{abstract}
A class of generalized Schr\"{o}dinger problems in bounded domain is studied. A complete overview of the set of solutions is provided, depending on the values assumed by parameters involved in the problem. In order to obtain the results, we combine monotony, bifurcation and variational methods.   

\end{abstract}
\maketitle

\section{Introduction}


In this work we investigate general conditions for which the stationary generalized Schr\"{o}dinger problem
\begin{equation}\label{P}\tag{$P_{\lambda, q}$}
\left \{ \begin{array}{ll}
-div( \vartheta(u)\nabla u)+\frac{1}{2}\vartheta'(u)|\nabla u|^{2}= \lambda |u|^{q-1}u & \mbox{in $\Omega$,}\\
u=0 & \mbox{on $\partial\Omega$,}
\end{array}\right.
\end{equation}
has nontrivial solutions, where $\Omega\subset\R^{N}$, $N\geq 3$, is a bounded smooth domain, $q>0$, $\lambda$ is a real parameter and $\vartheta:\R\to[1,\infty)$ is an even $C^{1}$-function satisfying some suitable hypotheses which will be stated later on.

\medskip

Choosing $\vartheta(s)=1+(l(s^{2})')^{2}/2$, for some $C^{2}$-function $l$, the problem \eqref{P} becomes 
\begin{equation}\label{schro}
\left \{ \begin{array}{ll}
-\Delta u-\Delta(l(u^{2}))l'(u^{2})u= \lambda |u|^{q-1}u & \mbox{in $\Omega$,}\\
u=0 & \mbox{on $\partial\Omega$.}
\end{array}\right.
\end{equation}
When $\Omega=\R^{N}$, \eqref{schro} is related to the existence of solitary wave solutions for the parabolic quasilinear Schr\"{o}dinger equation
\begin{equation}\label{evolution}
i\partial_{t}z=-\Delta z+V(x)z-\rho(|z|^{2})z-\Delta(l(|z|^{2}))l'(|z|^{2})z, \ x\in\R^{N},
\end{equation}
where $z:\R\times\R^{N}\to\mathds{C}$, $V:\R^{N}\to\R$ is a given potential and $l, \rho$ are real functions. Equation \eqref{evolution} appears naturally as a model for several physical phenomena, depending on the type of function $l$ considered. In fact, if $l(s)=s$, \eqref{evolution} describes the behavior of a superfluid film in plasma physics, see \cite{Kur}. For $l(s)=(1+s)^{1/2}$, \eqref{evolution} models the self-channeling of a high-power ultrashort laser in matter, see \cite{BG, BMMLB, CS, LSS}. Furthermore, \eqref{evolution} also appears in plasma physics and fluid mechanics \cite{LS}, in dissipative quantum mechanics \cite{Has}, in the theory of Heisenberg ferromagnetism and magnons \cite{QC} and in condensed matter theory \cite{MF}. 

In the last years, many authors have studied stationary Schr\"{o}dinger problems like \eqref{schro}, when $l(s)=s$ and $\Omega=\R^{N}$. In our best knowledge, the first result is due to \cite{PSW} which, by using a constrained minimization argument, proved the existence of nonnegative solutions for $\lambda>0$ large enough and $q\in (1, (N+2)/(N-2))$. Afterwards, a general existence result was derived \cite{LWW}. In \cite{LWW} the authors make a change of variable and reduce the quasilinear problem to a semilinear one and an Orlicz space framework was used to prove the existence of a positive solution via Mountain pass theorem. The same method of changing of variables was also used in
\cite{CJ}, but a framework involving usual Sobolev spaces was considered to treat the problem. More recent references can be found in \cite{DMS, DMO, DPY, SV, YWA}. In the case $l(s)=(1+s)^{1/2}$ fewer results are known, we refer the reader to \cite{DHS, ShW1, ShW2}.

In the present paper we are interested in investigating general conditions on $\vartheta$, in order to ensure the existence of nontrivial solutions for the problem \eqref{P}. The assumptions we are going to consider on the function $\vartheta$ are the following:
\begin{enumerate}
\item[$(\vartheta_{1})$] $s\mapsto \vartheta(s)$ is decreasing in $(-\infty, 0)$ and increasing in $(0, \infty)$;
\item[$(\vartheta_{2})$] $s\mapsto \vartheta(s)/s^{2}$ nondecreasing in $(-\infty, 0)$ and nonincreasing in $(0, \infty)$;
\item[$(\vartheta_{3})$] $\lim_{|s|\to \infty}\vartheta(s)/s^{2}=\alpha^{2}/2$, for some $\alpha>0$.
\end{enumerate}

Some simple examples of functions satisfying $(\vartheta_{1})-(\vartheta_{3})$ are: 
$$
\vartheta_{1}(s)=1+s^{2}, \ \vartheta_{2}(s)=(1+|s|^{p})^{1/p}+s^{2} \ \mbox{with $p\in [1, 2]$ and} \ \vartheta_{3}(s)=1+\ln(1+e^{s^{2}}).
$$

In this way, the present paper provides an unified approach to treat simultaneously a wide range of problems, among them some very relevant problems in terms of applications, which has been attacked separately in the literature. Other examples are given by:
$$
\vartheta_{4}(s)=1+\ln(e^{s\arctan s}+e^{s^{2}+s\arctan s}) \ \mbox{and} \ \vartheta_{5}(s)=1+\ln((1+|s|)^{|s|}(1+e^{s^{2}})).
$$

Under the stated assumptions, by consider different values in $\lambda$ and $q>0$, we provide a complete overview about the set of solutions of \eqref{P}. In our best knowledge, these results are new in this context, and some of them extend those obtained in \cite{FSS} to different general classes of Schr\"{o}dinger problems.

Our main results are as follows:

\medskip

\begin{theorem}\label{teor1}
Suppose $\vartheta$ satisfies $(\vartheta_{1})-(\vartheta_{3})$ and $q\in (0, 1)$. The following claims hold:
\begin{enumerate}
\item[$(a)$] \eqref{P2} has a unique positive solution $u_{\lambda}$ if, and only if, $\lambda>0$. Moreover, 
$$
\lim_{\lambda\to 0}|u_{\lambda}|_{\infty}=0 \ \mbox{and} \ \lim_{\lambda\to\infty}|u_{\lambda}|_{\infty}=\infty.
$$

\item[$(b)$] If $\lambda>0$, then \eqref{P2} has a sequence $\{u_{k}\}$ of sign-changing solutions such that
$$
\lim_{k\to\infty}\|u_{k}\|_{H_{0}^{1}(\Omega)}=0.
$$
\end{enumerate}
\end{theorem}

\begin{theorem}\label{teor2}
Suppose $(\vartheta_{1})-(\vartheta_{2})$ hold and $q=1$. Then, the problem \eqref{P2} possesses a unique positive solution $u_{\lambda}$ if, and only if, $\lambda>\vartheta(0)\lambda_{1}$. Moreover, 
$$
\lim_{\lambda\to\vartheta(0)\lambda_{1}}|u_{\lambda}|_{\infty}=0 \ \mbox{and} \ \lim_{\lambda\to\infty}|u_{\lambda}|_{\infty}=\infty.
$$
\end{theorem}

\begin{theorem}\label{teor3}
Suppose $\vartheta$ satisfies $(\vartheta_{1})-(\vartheta_{3})$ and $q\in (1, 3)$. The following claims hold:
\begin{enumerate}
\item[$(a)$] there exists $\lambda_{\ast}>0$ such that:
\begin{enumerate} 
\item[$(i)$] \eqref{P2} has no positive solution if $\lambda\in (0,\lambda_{\ast})$;
\item[$(ii)$] \eqref{P2} has at least one positive solution if $\lambda=\lambda_{\ast}$;
\item[$(iii)$] \eqref{P2} has at least two ordered positive solutions $w_{\lambda}<v_{\lambda}$, if $\lambda\in (\lambda_{\ast}, \infty)$. Moreover, the map $\lambda\mapsto v_{\lambda}$ is increasing and 
$$
\lim_{\lambda\to\infty}|v_{\lambda}|_{\infty}=\infty.
$$
\end{enumerate}
\item[$(b)$] For each $k\in \N$ there exists $\lambda_{k}>0$ such that \eqref{P2} has at least $k$ pairs of nontrivial solutions with negative energy, whatever $\lambda>\lambda_{k}$.
\end{enumerate}
\end{theorem}

\begin{theorem}\label{teor4}
Suppose the function $\vartheta$ satisfies $(\vartheta_{1})-(\vartheta_{3})$ and $q=3$. Then, \eqref{P2} has at least one positive solution if, and only if, $\lambda>(\alpha^{2}/4)\lambda_{1}$. Moreover,
$$
\lim_{\lambda\to (\alpha^{2}/4)\lambda_{1}}|u_{\lambda}|_{\infty}=\infty.
$$
\end{theorem}

\begin{theorem}\label{teor5}
Suppose the function $\vartheta$ satisfies $(\vartheta_{1})-(\vartheta_{3})$ and $q\in (3, 22^{\ast}-1)$, where $2^{\ast}=2N/(N-2)$. The following claims hold:
\begin{enumerate}
\item[$(a)$]  \eqref{P2} has at least one positive solution if, and only if, $\lambda>0$. Moreover,
$$
\lim_{\lambda\to 0}|u_{\lambda}|_{\infty}=\infty.
$$

\item[$(b)$] \eqref{P2} has infinitely many solutions with high energy, for each $\lambda>0$.
\end{enumerate}
\end{theorem}

\begin{theorem}\label{teor6}
Suppose the function $\vartheta$ satisfies $(\vartheta_{1})-(\vartheta_{3})$, $q\in [22^{\ast}-1, \infty)$ and $\Omega$ is a starshaped domain. Then, \eqref{P2} has no positive solution.
\end{theorem}

In the sequel we fix some notation which will be used along the paper: $\lambda_{k}$ denotes the $k$-th eigenvalue of the laplacian operator with homogeneous Dirichlet boundary condition. The unique eigenfunction associated to $\lambda_{k}$ and normalized in $L^{\infty}(\Omega)$ will be denoted by $\varphi_{k}$. The function $e$ is the unique solution of the problem
\begin{equation}\label{prob}
\left \{ \begin{array}{ll}
-\Delta e=1 & \mbox{in $D$,}\\
e=0 & \mbox{on $\partial D$,}
\end{array}\right.
\end{equation}
for some bounded and smooth domain $D\supset \overline{\Omega}$. Moreover, $e_{L}:=\min_{x\in\overline{\Omega}}e(x)>0$ and $e_{M}:=\max_{x\in\overline{\Omega}}e(x)$. The same letter $C$ stands for different positive constants whose exact value is irrelevant.

The paper is organized as follows. 

In Section \ref{se:eigen} we study a suitable change of variable which becomes problem \eqref{P} in a more manageable one. 
In Section \ref{se:trivial} we prove the main theorems of the paper.

\section{The dual formulation} \label{se:eigen}

In this section our main goal is to show that one can switch the task to look for solutions of the general semilinear problem
\begin{equation}\label{SP}\tag{SP}
\left \{ \begin{array}{ll}
-div( \vartheta(u)\nabla u)+\frac{1}{2}\vartheta'(u)|\nabla u|^{2}=p(u) & \mbox{in $\Omega$,}\\
u=0 & \mbox{on $\partial\Omega$,}
\end{array}\right.
\end{equation}
for the task to find solutions of 
\begin{equation}\label{DP}\tag{DP}
\left\{
\begin{array}{ll}
-\Delta v=f'(v)p(f(v)) & \mbox{in $\Omega$,}\\
v=0 & \mbox{on $\partial\Omega$,}
\end{array}\right.
\end{equation}
where $f\in C^{2}(\R)$ is a solution of the ordinary differential equation
\begin{equation}\label{ODE}\tag{ODE}
f'(s)=\frac{1}{\vartheta(f(s))^{1/2}} \ \mbox{for} \ s>0 \ \mbox{and} \ f(0)=0,
\end{equation}
with $f(s)=-f(-s)$ for $s\in (-\infty, 0)$.   

\medskip

Next proposition plays an important role throughout the paper.

\begin{proposition} \label{pp2}
Let $\vartheta\in C^{1}(\R)$ and $f$ a solution of \eqref{ODE}. The following claims hold:
\begin{enumerate}
\item[$(i)$] $f$ is uniquely defined and it is an increasing $C^{2}$-diffeomorphism, with $f''(s)=-\vartheta'(f(s))/2\vartheta(f(s))^{2}$, for all $s>0$;
\item[$(ii)$] $0< f'(s)\leq 1$, for all $s\in\R$;
\item[$(iii)$]  $\lim_{s\to 0}f(s)/s=1/\vartheta(0)^{1/2}$; 
\item[$(iv)$]  $ |f(s)|\leq |s|$, for all $s\in\R$; 
\item[$(v)$] Suppose $(\vartheta_{1})-(\vartheta_{2})$ hold. Then, $|f(s)|/2\leq f'(s)|s|\leq |f(s)|$, for all $s\in\R$,
and the map $s\mapsto |f(s)|/\sqrt{|s|}$ is nonincreasing in $(-\infty, 0)$ and nondecreasing in $(0, \infty)$; 
\item[$(vi)$] Suppose that $(\vartheta_{1})-(\vartheta_{3})$ hold. Then,
$$
\dis\lim_{|s|\to\infty}\frac{|f(s)|}{\sqrt{|s|}}=\left(\frac{8}{\alpha^{2}}\right)^{1/4} \ \mbox{and} \ \lim_{|s|\to\infty}\frac{f(s)}{s}=0,
$$
where $\alpha$ is given in $(\vartheta_{3})$. Moreover, there exists $C>0$ such that 
$$
|f'(s)f(s)|\leqslant C, \ \forall \ s\in\R.
$$ 

\end{enumerate}
\end{proposition}

\begin{proof}
$(i)$-$(ii)$ Existence, uniqueness, regularity, monotonicity and $(ii)$ follow directly from \eqref{ODE}.  To see that $f(\R)=\R$, observe that $f(s)=(\Upsilon^{-1})(s)$, where
$$
\Upsilon(t)=\int_{0}^{t}\vartheta(r)^{1/2}dr.
$$
Since $\vartheta\geq 1$, $|\Upsilon(t)|\geq |t|$ for all $t\in\R$. Consequently, $\lim_{|t|\to\infty}|\Upsilon(t)|=\infty$. Thence, $\lim_{|s|\to\infty}|f(s)|=\infty$. 

$(iii)$ Notice that, from L'H\^{o}spital rule, we get
$$
\dis\lim_{s\to 0}\frac{f(s)}{s}=\lim_{s\to 0}f'(s)=\frac{1}{\vartheta(0)^{1/2}}.
$$

$(iv)$ It follows from $(ii)$. $(v)$ Since $f$ is odd and $\vartheta$ is even, it is sufficient to prove the inequalities for $s>0$. For that, let $r_{1}:[0,\infty)\to\R$ defined by
\begin{equation*}
r_{1}(s)=f(s)\vartheta(f(s))^{1/2}-s.
\end{equation*}
Notice that $r_{1}(0)=0$ and, by \eqref{ODE} and $(\vartheta_{1})$, we have
$$
r_{1}'(s)=\vartheta'(f(s))f(s)/2\vartheta(f(s))>0.
$$
The second inequality in item $(v)$ is a direct consequence of the previous inequality. Now, to prove the first inequality in $(v)$, let $r_{2}:[0,\infty)\to\R$ defined by
\begin{equation*}
r_{2}(s)=2s-f(s)\vartheta(f(s))^{1/2}.
\end{equation*}
We have that $r_{2}(0)=0$ and, by \eqref{ODE} and $(\vartheta_{2})$,
$$
r_{2}'(s)=1-\vartheta'(f(s))f(s)/2\vartheta(f(s))\geq 0,
$$
showing that the inequality in $(v)$ holds. Moreover, since
$$
\left(\frac{f(s)}{\sqrt{s}}\right)'=\frac{2f'(s)s-f(s)}{2s\sqrt{s}}\geq 0, \ \forall \ s>0,
$$
the second part of $(v)$ follows.

$(vi)$ Observe that from $(v)$, we have 
$$
\lim_{|s|\to \infty}\frac{f(s)}{\sqrt{|s|}}=l, \ \mbox{with $l\in (0,\infty]$}.
$$ 
Again, since $f$ is odd and $\vartheta$ is even, it is sufficient to consider the case $s\to\infty$. Suppose that 
\begin{equation}\label{52}
\lim_{s\to \infty}f(s)/\sqrt{s}=\infty.
\end{equation} 
If this is the case then, by $(i)$, we get $f(s)\to\infty$ as $s\to\infty$. By applying the L'H\^{o}spital rule and using $(\vartheta_{3})$, we conclude from \eqref{52}, that
\begin{eqnarray*}
\lim_{s\to\infty}\frac{f(s)}{\sqrt{s}}&=&\dis\lim_{s\to\infty}2f'(s)\sqrt{s}\\
&=&2\dis\lim_{s\to\infty}\sqrt{\frac{s}{\vartheta(f(s))}}\\
&=&2\sqrt{\frac{\dis\lim_{s\to\infty}\left(\sqrt{s}/f(s)\right)^{2}}{\dis\lim_{s\to\infty}\vartheta(f(s))/f(s)^{2}}}\\
&=&2\sqrt{\frac{0}{(\alpha^{2}/2)}}=0.
\end{eqnarray*}
Showing that 
\begin{equation}\label{51}
\lim_{s\to\infty}f(s)/\sqrt{s}=0.
\end{equation}
Since \eqref{51} contradicts \eqref{52}, it follows that $0<\lim_{s\to\infty}f(s)/\sqrt{s}=l<\infty$. Applying one more time the L'H\^{o}spital rule, we have
$$
l=2\sqrt{\frac{\dis\lim_{s\to\infty}\left(\sqrt{s}/f(s)\right)^{2}}{\dis\lim_{s\to\infty}\vartheta(f(s))/f(s)^{2}}}=2\sqrt{\frac{1/l^{2}}{(\alpha^{2}/2)}}.
$$
Or equivalently,
\begin{equation}\label{53}
l=\left(\frac{8}{\alpha^{2}}\right)^{1/4}. 
\end{equation}
On the other hand, from \eqref{53},
$$
\dis\lim_{s\to\infty}\frac{f(s)}{s}=\dis\lim_{s\to\infty}\frac{f(s)}{\sqrt{s}}\frac{1}{\sqrt{s}}=\left(\frac{8}{\alpha^{2}}\right)^{1/4}\times 0=0. 
$$  
Finally, from
$$
\lim_{s\to\infty}f'(s)f(s)=\lim_{s\to\infty}f'(s)\sqrt{s}\times\lim_{s\to\infty}\frac{f(s)}{\sqrt{s}}=\frac{\sqrt{2}}{l\alpha}\times l=\frac{\sqrt{2}}{\alpha},
$$
the second part of $(vi)$ follows.
\end{proof}

\begin{remark}\label{rem1}
It is a consequence of the item $(i)$ in Proposition \ref{pp2} that $f$ is positive in $(0,\infty)$ and negative in $(-\infty, 0)$. Moreover, the inverse, $f^{-1}$, of $f$ is also a $C^{2}$-function.
\end{remark}

\begin{proposition}\label{change}
A function $v\in C^{1}(\overline{\Omega})$ is a weak solution of \eqref{DP} if, and only if, $u=f(v)\in C^{1}(\overline{\Omega})$ is a weak solution of \eqref{SP}.
\end{proposition}

\begin{proof}
Let $v\in C^{1}(\overline{\Omega})$ be a weak solution of \eqref{DP}. It is clear that $u=f(v)\in C^{1}(\overline{\Omega})$. Moreover,
$$
\nabla u=f'(v)\nabla v \ \mbox{and} \ \nabla v=\frac{1}{f'(v)}\nabla u=(f^{-1})'(u)\nabla u.
$$

Since $v$ is a weak solution of \eqref{DP} and $u=f(v)$, we have
\begin{equation}\label{weso}
\int_{\Omega}(f^{-1})'(u)\nabla u\nabla w dx=\int_{\Omega}\frac{p(u)}{(f^{-1})'(u)}w dx, \ \forall \ w\in H_{0}^{1}(\Omega).
\end{equation}
Since $v\in C^{1}(\overline{\Omega})$, for each $\varphi\in H_{0}^{1}(\Omega)$, we can choose $w=(f^{-1})'(u)\varphi=\vartheta(v)^{1/2}\varphi\in H_{0}^{1}(\Omega)$. Moreover, by \eqref{weso}
$$
\int_{\Omega}(f^{-1})'(u)(f^{-1})''(u)|\nabla u|^{2}\varphi dx+ \int_{\Omega}[(f^{-1})'(u)]^{2}\nabla u\nabla \varphi dx=\int_{\Omega}p(u)\varphi dx.
$$
By, \eqref{ODE},
$$
-\frac{1}{2}\int_{\Omega}\vartheta'(u)|\nabla u|^{2}\varphi dx+ \int_{\Omega}\vartheta(u)\nabla u\nabla \varphi dx=\int_{\Omega}p(u)\varphi dx.
$$
Integrating by parts, we conclude that
$$
\int_{\Omega}\left[-div( \vartheta(u)\nabla u)+(1/2)\vartheta'(u)|\nabla u|^{2}\right]\varphi dx=\int_{\Omega}p(u)\varphi dx, \ \forall \ \varphi\in H_{0}^{1}(\Omega). 
$$
Showing that $u$ is a weak solution of \eqref{SP}. The reverse is analogous.
\end{proof}

\medskip

In view of the previous proposition, along of the paper we will interested in studying the problem \eqref{DP}, which is known as the dual problem associated to \eqref{SP}.

\section{Existence and asymptotic behavior of solutions}\label{se:trivial}
In this section we are going to study the problem
\begin{equation}\label{P2}\tag{$P_{\lambda, q}$}
\left \{ \begin{array}{ll}
-div( \vartheta^{2}(u)\nabla u)+\vartheta(u)\vartheta'(u)|\nabla u|^{2}=\lambda |u|^{q-1}u & \mbox{in $\Omega$,}\\
u=0 & \mbox{on $\partial\Omega$,}
\end{array}\right.
\end{equation}
where $\Omega\subset\R^{N}$ is bounded smooth domain, $\lambda$ is a real parameter and $q>0$. As we have mentioned before, the dual problem associated to \eqref{P2} is
\begin{equation}\label{P2'}\tag{$P_{\lambda, q}'$}
\left \{ \begin{array}{ll}
-\Delta v=\lambda g(v) & \mbox{in $\Omega$,}\\
v=0 & \mbox{on $\partial\Omega$,}
\end{array}\right.
\end{equation}
where $g(s):= f'(s)|f(s)|^{q-1}f(s)$ and the function $f$ is the solution of \eqref{ODE}. 

Next lemma provides us some properties of the function $g$ which play an important role in the proof of our main results.

\begin{lemma}\label{main}
The function $g$ has the following properties:
\begin{enumerate}
\item[$(i)$] $\lim_{s\to 0}g(s)/s=\infty$, if $q\in (0, 1)$;

\item[$(ii)$] $\lim_{s\to 0}g(s)/s=1/\vartheta(0)$, if $q=1$;

\item[$(iii)$] $\lim_{s\to 0}g(s)/s=0$, if $q\in (1, \infty)$;

\item[$(iv)$] Suppose $(\vartheta_{1})-(\vartheta_{3})$ hold. Then, $\lim_{|s|\to \infty}g(s)/s=0$, if $q\in (0, 3)$;

\item[$(v)$] Suppose $(\vartheta_{1})-(\vartheta_{3})$ hold. Then, $\lim_{|s|\to \infty}g(s)/s=4/\alpha^{2}$, if $q=3$;

\item[$(vi)$] Suppose $(\vartheta_{1})-(\vartheta_{2})$ hold. Then, $\lim_{|s|\to \infty}g(s)/s=\infty$, if $q\in (3, \infty)$;

\item[$(vii)$] Suppose $(\vartheta_{1})-(\vartheta_{3})$ hold. Then, the map $s\mapsto g(s)/s$ is decreasing with regard $|s|$, for $q\in (0, 1]$;

\item[$(viii)$] Suppose $(\vartheta_{2})$ holds. Then, the map $s\mapsto g(s)/s$ is increasing with regard $|s|$, for $q\in [3, \infty)$.

\end{enumerate} 
\end{lemma}

\begin{proof}
Items $(i)$, $(ii)$ and $(iii)$ are straightforward consequences of Proposition \ref{pp2}$(ii)-(iii)$, because in these cases
$$
\lim_{s\to 0}\frac{g(s)}{s}=\lim_{s\to 0}f'(s)\left(\frac{|f(s)|}{|s|}\right)^{q-1}\frac{f(s)}{s}|s|^{q-1}=\left \{ \begin{array}{ll}
\infty & \mbox{if $q\in (0, 1)$,}\\
1/\vartheta(0) & \mbox{if $q=1$,}\\
0 & \mbox{if $q\in (1, \infty)$.}
\end{array}\right.
$$

$(iv)$ Observe that by $(v)$ and $(vi)$ of Proposition \ref{pp2}, we obtain
$$
\frac{g(s)}{s}=\frac{f'(s)|f(s)|^{q}}{|s|}\leq\frac{|f(s)|^{q+1}}{s^{2}}\leq \left(\frac{8}{\alpha^{2}}\right)^{(q+1)/4}\frac{1}{|s|^{(3-q)/2}}, 
$$
for all $s\in\R\backslash\{0\}$ and $q\in (0, 3)$. Thus, 
$$
\lim_{|s|\to\infty}\frac{g(s)}{s}=0, 
$$
for $q\in (0, 3)$. 
Therefore, item $(iv)$ follows. 

$(v)$ It is sufficient to note that
$$
\lim_{|s|\to\infty}\frac{g(s)}{s}=\lim_{|s|\to\infty}\left(\frac{|f(s)|}{\sqrt{|s|}}\right)^{2}\times\lim_{|s|\to\infty}\frac{|f(s)|}{\vartheta(f(s))^{1/2}}=\left(\frac{8}{\alpha^{2}}\right)^{1/2}\times \frac{\sqrt{2}}{\alpha}=\frac{4}{\alpha^{2}}.
$$

$(vi)$ It is a consequence of Proposition \eqref{pp2}$(v)$ and of the inequality
$$
\frac{g(s)}{s}=\frac{f'(s)|f(s)|^{q}}{|s|}\geq \frac{|f(s)|^{q+1}}{2s^{2}}\geq \frac{f(1)^{q+1}}{2}|s|^{(q-3)/2}, \ \forall \ |s|\geq 1.
$$

$(vii)$ Since $g$ is odd, it sufficient to consider case $s>0$. Since,
$$
\left(\frac{g(s)}{s}\right)'=\frac{f''(s)f(s)^{q}s+qf(s)^{q-1}f'(s)^{2}s-f'(s)f(s)^{q}}{s^{2}},
$$
it follows from Proposition \ref{pp2}$(v)$ that, for $q\in (0, 1]$,
$$
\left(\frac{g(s)}{s}\right)'<\frac{f'(s)f(s)^{q-1}}{s^{2}}\left(qf'(s)s-f(s)\right)\leq 0, \ \forall \ s\in (0, \infty).
$$

$(viii)$ Again, let us consider just case $s>0$. Note that 
\begin{eqnarray*}
\left(\frac{g(s)}{s}\right)'&=&\left(qf(s)^{q-1}f'(s)s-f(s)^{q}-\frac{\vartheta'(f(s))f(s)^{q}s}{2\vartheta(f(s))^{3/2}}\right)\frac{1}{s^{2}\vartheta(f(s))^{1/2}}.
\end{eqnarray*}
On the other hand, by $(\vartheta_{2})$ and Proposition \ref{pp2}$(v)$, for each $q\in [3, \infty)$, we have
\begin{eqnarray*}
qf(s)^{q-1}f'(s)s-f(s)^{q}-\frac{\vartheta'(f(s))f(s)^{q}s}{2\vartheta(f(s))^{3/2}}&\geq& qf(s)^{q-1}f'(s)s-f(s)^{q}-f(s)^{q-1}f'(s)s\\
&=&f(s)^{q-1}\left((q-1)f'(s)s-f(s)\right)\\
&\geq& f(s)^{q-1}\left(2f'(s)s-f(s)\right)>0,
\end{eqnarray*}
for all $s\in (0, \infty)$. Therefore, $s\mapsto g(s)/s$ is increasing in $(0,\infty)$ for all $q\in [3, \infty)$.
\end{proof}

Before proving the existence results we state the following proposition, which justifies why, then, we are going to consider positive values for the parameter $\lambda$.

\begin{proposition}
If $\lambda\leq 0$, then problem \eqref{P2} has no nontrivial solution.
\end{proposition}

\begin{proof}
Let $u$ a solution of $\eqref{P2}$ with $\lambda\leq 0$. Then, by Proposition \ref{change}, $v=f^{-1}(u)$ is a solution of $\eqref{P2'}$. Thus,
$$
0\leq \|v\|^{2}=\lambda\int_{\Omega}g(v)v dx=\lambda\int_{\Omega}f'(v)|f(v)|^{q-1}f(v)v dx\leq 0.
$$
Showing that $v=0$ and, consequently, $u=0$.
\end{proof}

Now on, even if nothing is said, we are considering $\lambda>0$.



\subsection{Case $0<q<1$}.\\

We are ready to prove the main results of this subsection.

\medskip

{\bf Proof of Theorem \ref{teor1}(a):}

\medskip

For each $\lambda>0$, by Lemma \ref{main}$(i)$, there exists $\varepsilon>0$ such that
\begin{equation}\label{sub}
\frac{\lambda_{1}}{\lambda}\leq \frac{g(\varepsilon \varphi_{1})}{\varepsilon \varphi_{1}}.
\end{equation} 
Choosing $\underline{v}:=\varepsilon \varphi_{1}$, it follows from
\begin{equation}\label{sub'}
-\Delta\underline{v}=\varepsilon\lambda_{1}\varphi_{1}<\lambda g(\varepsilon\varphi_{1})=\lambda g(\underline{v}) \ \mbox{in} \ \Omega,
\end{equation}
that $\underline{v}$ is a sub-solution of \eqref{P2'}. On the other hand, choosing $\overline{v}:=Ke$ where $e$ is defined in (\ref{prob}) with $K$ being a positive constant, which is large enough, then $\overline{v}$ is a super-solution. Indeed, by Lemma \ref{main}$(iv)$, there exists $K>0$ large enough such that
$$
e_{M}\frac{g(K e_{L})}{K e_{L}}\leq\frac{1}{\lambda}.
$$
Thus,
\begin{equation}\label{sup}
-\Delta\overline{v}=K\geq\lambda Ke_{M}\frac{g(K e_{L})}{K e_{L}}\geq \lambda g(\overline{v}) \ \mbox{in} \ \Omega.
\end{equation}
Showing that $\overline{v}$ is a super-solution of \eqref{P2'}. Choosing $\varepsilon$ smaller and $K$ greater, if it is necessary, we can assume that $\underline{v}\leq\overline{v}$. 
Consequently, we conclude the existence of a positive classical solution $v_{\lambda}$ of \eqref{P2'} such that
\begin{equation}\label{ordered}
\underline{v}\leq v_{\lambda}\leq \overline{v}.
\end{equation}

The uniqueness of positive solution follows from Lemma \ref{main}$(vii)$ and \cite{BO}. Finally, if $\lambda\to 0$, we can choose $K=K(\lambda)\to 0$ in \eqref{sup} to conclude (by \eqref{ordered}) that
$$
\lim_{\lambda\to 0}|v_{\lambda}|_{\infty}=0.
$$
In the same way, if $\lambda\to\infty$, we can choose $\varepsilon=\varepsilon(\lambda)\to \infty$ in \eqref{sub} to obtain (by \eqref{ordered} again)
$$
\lim_{\lambda\to \infty}|v_{\lambda}|_{\infty}=\infty.
$$
$\square$

\medskip

{\bf Proof of Theorem \ref{teor1}(b):}

\medskip

Let $I:H_{0}^{1}(\Omega)\to\R$ be defined by
$$
I(v)=\frac{1}{2}\|v\|^{2}-\lambda\int_{\Omega}G(v)dx
$$
the energy functional of \eqref{P2'}, where $G(s)=\int_{0}^{s}g(t)dt=[1/(q+1)]|f(s)|^{q+1}$. Since $f\in C^{2}(\R)$, it follows that $I\in C^{1}(H_{0}^{1}(\Omega), \R)$ and
$$
I'(v)w=\int_{\Omega}\nabla v\nabla w dx-\lambda\int_{\Omega} g(v)w dx, \ \forall \ v, w\in H_{0}^{1}(\Omega).
$$
Clearly $I$ is even and $I(0)=0$. Moreover, $I$ is coercive and bounded from below, because by item $(iv)$ of Proposition \ref{pp2}, we have
$$
I(v)\geq \left(\frac{1}{2}\|v\|^{1-q}-C\lambda\right)\|v\|^{q+1}.
$$
Furthermore, $I$ satisfies the Palais-Smale condition in any level $c\in\R$. Indeed, if $I(v_{n})\to c$ and $I'(v_{n})\to 0$ in $\vartheta^{-1}(\Omega)$, then, by coercivity of $I$, it follows that $\{\|v_{n}\|\}$ is bounded. Thence, there exists $v_{0}\in H_{0}^{1}(\Omega)$ such that
\begin{equation}\label{agora}
v_{n}\rightharpoonup v_{0} \ \mbox{in $H_{0}^{1}(\Omega)$}
\end{equation}
and
\begin{equation}\label{sim}
v_{n}\to v_{0} \ \mbox{in $L^{s}(\Omega)$}, \ \forall \ s\in [1, 2^{\ast}).
\end{equation}

Since $\{\|v_{n}\|\}$ is bounded, we can use \eqref{sim} and compact embedding to obtain
$$
o_{n}(1)=I'(v_{n})v_{n}=\frac{1}{2}\|v_{n}\|^{2}+\lambda\int_{\Omega}g(v_{0})v_{0}dx+o_{n}(1).
$$
Proving that
\begin{equation}\label{aca}
\|v_{n}\|^{2}=-2\lambda\int_{\Omega}g(v_{0})v_{0}dx+o_{n}(1).
\end{equation} 

In the same way, by using \eqref{agora} and compact embedding, we get
$$
o_{n}(1)=I'(v_{n})v_{0}=\frac{1}{2}\|v_{0}\|^{2}+\lambda\int_{\Omega}g(v_{0})v_{0}dx+o_{n}(1).
$$
Showing that
\begin{equation}\label{bou}
\|v_{0}\|^{2}=-2\lambda\int_{\Omega}g(v_{0})v_{0}dx.
\end{equation}

Comparing \eqref{aca} and \eqref{bou}, we conclude that
$$
\|v_{n}\|^{2}\to\|v_{0}\|^{2}
$$
and, consequently, $v_{n}\to v_{0}$ in $H_{0}^{1}(\Omega)$. Thus, $I$ satisfies the Palais-Smale condition. 

\medskip

Finally, by Proposition \ref{pp2}$(iii)$, there exists a positive constant $C$ such that
$$
|f(s)|\geq C|s|, \ \forall \ |s|\leq 1.
$$
Thus, for each $k\in\N$, let $X_{k}=Span\{f_{1}, \ldots, f_{k}\}$ be a $k$-dimensional subspace of $H_{0}^{1}(\Omega)\cap L^{\infty}(\Omega)$ such that $f_{1}, \ldots, f_{k}$ are two-by-two orthogonals in $H_{0}^{1}(\Omega)$ and $|f_{i}|_{\infty}\leq 1$, for all $i\in \{1, \ldots, k\}$. Clearly, by choosing $0<\rho_{k}<\min_{1\leq i\leq k}\|f_{i}\|/k$, it follows that if $v\in X_{k}$ and $\|v\|=1$, then $|\rho_{k}v|_{\infty}\leq 1$. Thence,

\begin{eqnarray*}
I(\rho_{k} v)&=&\frac{\rho^{2}_{k}}{2}-\lambda\int_{\Omega}G(\rho_{k} v)dx\\
&\leq & \frac{\rho^{2}_{k}}{2}-\frac{\lambda C_{1}\rho^{q+1}_{k}}{q+1}\int_{\Omega}|v|^{q+1}dx.
\end{eqnarray*}
Since $X_k$ is a finite dimensional subspace, there exists $C_{k}>0$ such that
$$
C_{k}\|v\|^{q+1}\leq \int_{\Omega}|v|^{q+1}dx, \ \forall \ v\in X_{k},
$$
we get
\begin{equation}\label{bom}
I(\rho_{k} v)\leq \frac{1}{2}\rho^{2}_{k}-\frac{\lambda C_{1}C_{k}}{q+1}\rho^{q+1}_{k}.
\end{equation}
Since $q\in (0, 1)$, we can choose $\rho_{k}>0$ even lower in \eqref{bom} to conclude that
$$
\sup_{v\in S_{\rho_{k}}\cap X_{k}}I(v)<0,
$$
where $S_{\rho_{k}}=\{v\in H_{0}^{1}(\Omega): \|v\|=\rho_{k}\}$. The result follows now by Theorem 1 in \cite{Ka}.
$\square$


\subsection{Case $q=1$}.\\

{\bf Proof of Theorem \ref{teor2}:}

\medskip

Let $v$ a positive solution of \eqref{P2'}. Then,
$$
0=\mu_{1}\left(-\Delta-\lambda \frac{g(v)}{v}\right),
$$
where $\mu_{1}\left(-\Delta-\lambda g(v)/v\right)$ is the principal eigenvalue of the problem
\begin{equation}
\left \{ \begin{array}{ll}
-\Delta v-\lambda g(v)=\mu v & \mbox{in $\Omega$,}\\
v=0 & \mbox{on $\partial\Omega$.}
\end{array}\right.
\end{equation}
It follows from $(ii)$ and $(vii)$ in Lemma \ref{main} that 
$$
\frac{g(s)}{s}<\frac{1}{\vartheta(0)}, \ \forall \ s>0.
$$
Thereby,
$$
0=\mu_{1}\left(-\Delta-\lambda \frac{g(v)}{v}\right)>\mu_{1}\left(-\Delta-\frac{\lambda}{\vartheta(0)}\right)=\lambda_{1}-\frac{\lambda}{\vartheta(0)},
$$
Therefore, if there exists positive solution of \eqref{P2'}, then $\lambda>\vartheta(0)\lambda_{1}$.

If $\lambda>\vartheta(0)\lambda_{1}$, we can use $(ii)$, $(iv)$ and $(vii)$ in Lemma \ref{main} and to argue exactly in the same way as in the proof of Theorem \ref{teor1} to prove that $\underline{v}:=\varepsilon \varphi_{1}$, with $\varepsilon$ small enough, is a sub-solution, $\overline{v}:=Ke$, with $K$ large enough, is a super-solution and \eqref{P2'} admits a unique solution $v_{\lambda}$, which satisfies
\begin{equation}\label{inn}
\varepsilon\varphi_{1}\leq v_{\lambda}\leq Ke \ \mbox{in} \ \Omega.
\end{equation}

Then, $u_{\lambda}=f^{-1}(v_{\lambda})$ is the unique solution of \eqref{P2}. Finally, we can choose, in \eqref{sub}, $\varepsilon(\lambda)$ such that $\varepsilon(\lambda)\to\infty$ as $\lambda\to\infty$, thus
$$
v_{\lambda}(x)\geq \varepsilon(\lambda)\varphi_{1}(x)\to\infty.
$$
Since the inverse $f^{-1}$ is an increasing diffeomorphism in $\R$, see Proposition \ref{pp2}$(i)$, we conclude also that 
$$
u_{\lambda}(x)\to\infty
$$
and, therefore, $|u_{\lambda}|_{\infty}\to\infty$ as $\lambda\to\infty$. 

Finally, observe that by (\ref{inn}) $v_\l$ is bounded in $L^\infty(\O)$ as $\lambda\to\vartheta(0)\lambda_{1}$. By the elliptic regularity and a bootstrapping-argument, we can conclude that  $v_\l$ is bounded in $C^{2,\alpha}(\overline\O)$, $\alpha\in (0,1)$. Then,  $v_\l\to v_0\geq 0$ in $C^2(\ov\O)$ as $\l\to\vartheta(0)\lambda_{1}$. Since we have proved that the unique solution for $\l=\vartheta(0)\lambda_{1}$ is the trivial one, we get that $v_0\equiv 0$ in $\O$. 
$\square$


\subsection{Case $1<q<3$}.\\

Next lemmas will be used in the proof of our main results for $q\in (1, 3)$.

\begin{lemma}\label{bounde}
Suppose $(\vartheta_{1})-(\vartheta_{3})$ hold and $q\in (1, 3)$. If there exists a positive solution $v_{\lambda}$ of \eqref{P2'}, then
\begin{equation}\label{add1}
v_{\lambda}\leq C\psi,
\end{equation}
where 
\begin{equation}\label{add2}
C=\lambda^{2/(3-q)}\left(\frac{8}{\alpha^{2}}\right)^{(q+1)/2(3-q)}
\end{equation}
and $\psi$ is the unique solution of the problem 
\begin{equation}
\label{new}
\left \{ \begin{array}{ll}
-\Delta w=w^{(q-1)/2} & \mbox{in $\Omega$,}\\
w>0 & \mbox{in $\Omega$,}\\
w=0 & \mbox{on $\partial\Omega$.}
\end{array}\right.
\end{equation}
\end{lemma}

\begin{proof}
Firstly, observe that since $q<3$ then $(q-1)/2<1$, and hence (\ref{new}) possesses a unique positive solution.

By items $(v)$ and $(vi)$ of Proposition \ref{pp2}, we have
\begin{equation}\label{ineq10}
|f(s)|\leq \left(\frac{8}{\alpha^{2}}\right)^{1/4}\sqrt{|s|}, \ \forall s\in\R.
\end{equation}
Consequently, from Proposition \ref{pp2}$(v)$ and \eqref{ineq10}, if $v_{\lambda}$ is a positive solution of \eqref{P2'}, then
$$
-\Delta v_{\lambda}\leq \lambda \left(\frac{8}{\alpha^{2}}\right)^{(q+1)/4} v_{\lambda}^{(q-1)/2}.
$$
Thus, $v_{\lambda}$ is a sub-solution of the problem
\begin{equation}
\left \{ \begin{array}{ll}
-\Delta w=\lambda (8/\alpha^{2})^{(q+1)/4} w^{(q-1)/2} & \mbox{in $\Omega$,}\\
w>0 & \mbox{in $\Omega$,}\\
w=0 & \mbox{on $\partial\Omega$,}
\end{array}\right.
\end{equation}
which has a unique solution 
$$
\overline{w}=\lambda^{2/(3-q)} (8/\alpha^{2})^{(q+1)/2(3-q)}\psi,
$$ 
because $q\in (0, 3)$. It follows from Lemma 3.3 in \cite{ABC} that
$$
v_{\lambda}\leq C\psi,
$$
where
$$
C=\lambda^{2/(3-q)} (8/\alpha^{2})^{(q+1)/2(3-q)}.
$$
\end{proof}

\begin{lemma}\label{exis1}
Suppose $(\vartheta_{1})-(\vartheta_{3})$ hold and $q\in (1, 3)$. Then, there exists $\overline{\lambda}>0$ such that $\eqref{P2}$ has a positive solution, for all $\lambda\geq\overline{\lambda}$.
\end{lemma}

\begin{proof}
Since $\partial \varphi_{1}/\partial\eta<0$ on $\partial\Omega$, where $\eta$ is the outward unit normal vector on $\partial\Omega$, there exists a neighborhood $\Omega_{r}$ of $\partial\Omega$, for some $r>1$, such that
$$
(1-r)\varphi_{1}^{-2}|\nabla\varphi_{1}|^{2}+\lambda_{1}\leq 0 \ \mbox{in $\Omega_{r}$}.
$$
On the other hand, there exists $\overline{\lambda}>0$, such that
$$
r\frac{\varphi_{1}^{r}}{g(\varphi_{1}^{r})}\left( (1-r)\varphi_{1}^{-2}|\nabla\varphi_{1}|^{2}+\lambda_{1}\right)\leq \lambda \ \mbox{in $\Omega\backslash\overline{\Omega}_{r}$},
$$
for all $\lambda\geq \overline{\lambda}$. Showing that
$$
-\Delta(\varphi_{1}^{r})=r\varphi_{1}^{r}\left(\lambda_{1}+(1-r)\varphi_{1}^{-2}|\nabla\varphi_{1}|^{2}\right)\leq\lambda g(\varphi_{1}^{r}) \ \mbox{in $\Omega$}.
$$
Consequently $\underline{v}=\varphi_{1}^{r}$, for some $r>1$, is a sub-solution of \eqref{P2'}, if $\lambda\geq \overline{\lambda}$. On the other hand, it is a straightforward consequence of Lemma \ref{main}$(iv)$ that, for $K$ large enough, $\overline{v}=Ke$ is a super solution of \eqref{P2'}. Moreover, if appropriate, we can choose $K$ greater yet to ensure that $\underline{v}\leq \overline{v}$.

\end{proof}

\begin{lemma}\label{este}
Suppose $(\vartheta_{1})-(\vartheta_{3})$ hold and $q\in (1, 3)$. Then, there exists $\lambda_{\ast}>0$ such that \eqref{P2} has a positive solution  if, and only if, $\lambda\geq \lambda_{\ast}$. Moreover, there exists a maximal solution $\xi_{\lambda}$, for $\lambda\geq\lambda_{\ast}$, such that if $\nu>\mu\geq \lambda_{\ast}$, we have $\xi_{\lambda_{\ast}}\leq \xi_{\mu}<\xi_{\nu}$.
\end{lemma}

\begin{proof}
It follows from Lemma \ref{exis1} that $\Gamma=\{\lambda>0: \eqref{P2'}  \ \mbox{has a positive solution}\}\neq \emptyset$ and $\Gamma\subset (0, \infty)$. Let $\lambda_{\ast}:=\inf\Gamma$. If $\lambda>\lambda_{\ast}$, it is easy to see that, for any fixed $\lambda_{\ast}\leq\mu<\lambda$, the functions $\underline{v}=v_{\mu}$ and $\overline{v}=Ke$, for $K$ large enough, are ordered sub and super-solutions, respectively, for the problem \eqref{P2'}, where $v_{\mu}$ denotes a solution of \eqref{P2'} with $\lambda=\mu$.

When $\lambda=\lambda_{\ast}$, we take a sequence $\{\lambda_{n}\}\subset \Gamma$ such that $\lambda_{n}\downarrow\lambda_{\ast}$. Denote by $v_{n}$ a positive solution of \eqref{P2'} with $\lambda=\lambda_{n}$. By Lemma \ref{bounde}, it follows that $\{v_{n}\}\subset L^{\infty}(\Omega)$ and, by elliptic regularity, passing to a subsequence, $v_{n}\to v_{\ast}$ in $C^{2}(\overline{\Omega})$, where $v_{\ast}$ is a solution of \eqref{P2'} with $\lambda=\lambda_{\ast}$. Observe that $v_{\ast}\neq 0$ because, otherwise, we have $v_{n}\to 0$ in $C^{2}(\overline{\Omega})$. Since $q\in (1, 3)$, it follows from Proposition \ref{main}$(ii)$ that
$$
0=\lambda_{1}\left(-\Delta-\lambda_{n}\frac{g(v_{n})}{v_{n}}\right)\to\lambda_{1}.
$$
Last equality leads us to a contradiction.

To prove the existence of a maximal solution $\xi_{\lambda}$ for the problem \eqref{P2'}, observe that, by Lemma \ref{bounde}, if $v$ is a positive solution of $\eqref{P2'}$ then 
$$
v\leq C\psi,
$$
where $C$ is defined in \eqref{add2}, and $\psi$ is the unique solution of the problem 
\begin{equation}
\left \{ \begin{array}{ll}
-\Delta w=w^{(q-1)/2} & \mbox{in $\Omega$,}\\
w>0 & \mbox{in $\Omega$,}\\
w=0 & \mbox{on $\partial\Omega$.}
\end{array}\right.
\end{equation}

Now, taking $K>0$ such that the map $\gamma(s)=\lambda g(s)+Ks$ is increasing in $[0, C|\psi|_{\infty}]$ and considering the monotonic iteration 
$$
-\Delta v_{n+1}+K v_{n+1}=\gamma(v_{n}), v_{0}=\psi, v_{n+1}=0 \ \mbox{on $\partial\Omega$},
$$
we get a maximal solution in $[0, C|\psi|_{\infty}]$. Since, any positive solution $w$ of \eqref{P2'} satisfies $w<C\psi$, the existence of the maximal solution follows.

Finally, by arguing as previously, it follows that if $\nu>\mu\geq\lambda_{\ast}$, then there exists a positive solution $v$ of \eqref{P2'}, such that $\xi_{\nu}<v\leq Ke$, where $\xi_{\nu}$ is the maximal solution of the problem \eqref{P2'} with $\lambda=\nu$ and $Ke$ is a super solution of \eqref{P2'}, with $\lambda=\mu$. Consequently,
$$
\xi_{\nu}<v\leq \xi_{\mu},
$$
this completes the proof.

\end{proof}

The next lemma provides us some informations about the energy functional $I$ associated to the problem \eqref{P2'}.

\begin{lemma}\label{funct}
Suppose the function $\vartheta$ satisfies $(\vartheta_{1})-(\vartheta_{3})$ and $q\in (1, 3)$. Then, the following claims hold:
\begin{enumerate}
\item[$(i)$] $I$ is well defined, coercive and bounded from below;

\item[$(ii)$] there exists a subsequence, which we denote yet by $\{v_{n}\}$, such that $v_{n}\rightharpoonup v_{0}$ in $H_{0}^{1}(\Omega)$ and $I(v_{0})\leq \liminf_{n\to\infty}I(v_{n})$, whenever $\{I(v_{n})\}$ is bounded;

\item[$(iii)$] the origin of $H_{0}^{1}(\Omega)$ is a local minimum of $I$;
\end{enumerate}
\end{lemma}

\begin{proof}

$(i)$ By Proposition \ref{pp2}$(vi)$, we conclude that
\begin{equation}\label{bound1}
|G(s)|\leq C_{1}|s|+C_{2}|s|^{(q+1)/2}, \ \forall \ s\in\R.
\end{equation}
Since $q\in (1, 3)$, then $(q+1)/2\in (1, 2)$. Showing that $I$ is well defined. On the other hand, the inequality \eqref{bound1} also implies that $I$ is coercive and bounded from below.


$(ii)$ Since $\{I(v_{n})\}$ is bounded and $I$ is coercive, it follows that $\{v_{n}\}$ is bounded in $\vartheta^{1}_{0}(\Omega)$. Consequently, there exists $v_{0}\in H^{1}_{0}(\Omega)$ such that 
$$
v_{n}\rightharpoonup v_{0} \ \mbox{in $\vartheta^{1}_{0}(\Omega)$}
$$
and
$$
v_{n}\to v_{0} \ \mbox{in $L^{s}(\Omega)$, with} \ s\in [1, 2^{\ast}).
$$
Since 
$$
I(v_{0})-I(v_{n})=\frac{1}{2}\left(\|v_{0}\|^{2}-\|v_{n}\|^{2}\right)+\int_{\Omega}[G(v_{n})-G(v_{0})]dx
$$
and $G(v_{n})-G(v_{0})=g(\varepsilon_{n})(v_{n}-v_{0})$ with $\varepsilon_{n}$ between $v_{n}$ and $v_{0}$, we conclude that
$$
\int_{\Omega}[G(v_{n})-G(v_{0})]dx\to 0
$$
and the result follows.

$(iii)$ It is sufficient to note that, by Lemma \ref{main}$(iii)$, for each $\varepsilon>0$, there exists $C_{\varepsilon}>0$ such that
$|g(s)|\leq\varepsilon(s^{2}/2)+C_{\varepsilon}|s|^{p}$, for some $p\in (2, 2^{\ast})$. Thus,
$$
I(v)\geq \frac{1}{2}\left(1-\frac{\varepsilon \lambda}{\lambda_{1}}\right)\|v\|^{2}-\lambda C_{\varepsilon}\|v\|^{p+1}.
$$
Showing that $0$ is a local minimum.

\end{proof}

Now, we are ready to prove the main result of existence of positive solutions in the case $q\in (1, 3)$.

\medskip

{\bf Proof of Theorem \ref{teor3}(a):}

\medskip

$(i)$ Let $\lambda_{\ast}$ be as in Lemma \ref{este}. Clearly, there is no positive solution for $\lambda\in (0, \lambda_{\ast})$. $(ii)$ On the other hand, by Lemma \ref{este}, we know that for $\lambda=\lambda_{\ast}$, there exists a maximal positive solution $\xi_{\ast}:=\xi_{\lambda_{\ast}}$ of \eqref{P2'}. $(iii)$ It follows from Lemma \ref{funct}$(i)-(ii)$ that, for each $\lambda> \lambda_{\ast}$, there exists $v_{\lambda}\geq \xi_{\ast}$ such that
$$
I(v_{\lambda})=\min_{v\in\mathfrak{M}}I(v),
$$
where
$$
\mathfrak{M}=\{I(v): v\in H^{1}_{0}(\Omega) \ \mbox{and} \ v\geq \xi_{\ast}\}.
$$
Since $\xi_{\ast}$ is a sub-solution of \eqref{P2'} with $\lambda>\lambda_{\ast}$, it follows from strong maximum principle that $v_{\lambda}-\xi_{\ast}\in int (\mathfrak{N})$, where $\mathfrak{N}=\{v\in C^{1}_{0}(\overline{\Omega}): v\geq 0 \ \mbox{in $\Omega$}\}$. Therefore, $v_{\lambda}$ is a solution of \eqref{P2'}. Consequently, for $\lambda>\lambda_{\ast}$, $I$ admits two different minima, i.e., $v_{\lambda}$ and $0$.

We are going to prove that there exists a third minimum $0<w_{\lambda}<v_{\lambda}$ for $I$. For this, consider the closed and convex set
$$
\mathfrak{V}=\{v\in H^{1}_{0}(\Omega): 0\leq v\leq v_{\lambda}\}.
$$
We mean by a critical point of $I$ in $\mathfrak{V}$ to any $v\in \mathfrak{V}$ satisfying
$$
l(v)=\sup\{I'(v)(w-v), w\in \mathfrak{V} \ \mbox{and $\|w-v\|\leq 1$}\}=0.
$$
Since $0\in \mathfrak{V}$, it follows that $l(v)=0$ implies $I'(v)=0$. Note that $I$ satisfies the $(PS)_{c}$ condition in $\mathfrak{V}$, for any $c\in\R$. Indeed, let $\{v_{n}\}\subset \mathfrak{V}$ with $I(v_{n})\to c$ and $l(v_{n})\to 0$. Since $\{I(v_{n})\}$ is bounded and $I$ is coercive, 
$\square$

Next theorem improves the result of nonexistence obtained in the previous theorem, as well as it tells us that the higher the size of $\lambda$ the more solutions has the problem \eqref{P2}. Before, however, we need to prove a technical lemma.

\begin{lemma}\label{tec}
Let $\mathcal{K}\subset H_{0}^{1}(\Omega)\backslash\{0\}$ be a compact set which is symmetric with regard the origin of $H_{0}^{1}(\Omega)$. Then, there exist $\beta>0$ and $s_{\mathcal{K}}>0$ such that 
\begin{equation}\label{defi}
s\mathcal{K}:=\{sv:v\in\mathcal{K}\}\subset \mathcal{A}_{\beta}:=\{v\in H_{0}^{1}(\Omega): |[|v|>1]|\geq\beta\}, \ \forall \ s\geq s_{\mathcal{K}}.
\end{equation}
\end{lemma}

\begin{proof}
Indeed, otherwise, there exist $\{v_{n}\}\subset \mathcal{K}$ and $s_{n}\to\infty$ such that
\begin{equation}\label{conv0}
|[s_{n}|v_{n}|>1]|\to 0, \ \mbox{as $n\to\infty$}.
\end{equation}
Since $\mathcal{K}$ is compact, passing to a subsequence, there exists $v\in\mathcal{K}$ such that
$$
|v_{n}(x)|\to |v(x)| \ \mbox{a.e. in $\Omega$}.
$$
Since $v\neq 0$, there exist $\Omega_{0}\subset\Omega$ with positive measure and $\delta>0$ such that
$$
|v(x)|> \delta \ \mbox{in $\Omega_{0}$}.
$$
Consequently, there exist $n_{0}\in\N$ and  a subset $\hat{\Omega}_{0}\subset \Omega_{0}$ with positive measure, such that
$$
|v_{n}(x)|\geq\delta \ \mbox{in $\hat{\Omega}_{0}$}, \ \forall \ n\geq n_{0}.
$$
Thus, for $n$ large enough we have $\hat{\Omega}_{0}\subset [s_{n}|v_{n}|>1]$ and
$$
0<|\hat{\Omega}_{0}|\leq |[s_{n}|v_{n}|>1]|.
$$
The last inequality contradicts \eqref{conv0}. Showing that \eqref{defi} holds. 
\end{proof}

\medskip

{\bf Proof of Theorem \ref{teor3}(b):}

\medskip

It is clear that $I(0)=0$, $I$ is even and $C^1$. Moreover, by Lemma \ref{funct}$(i)$ we know that $I$ is coercive and bounded from below. On the other hand, by arguing as in case $q\in (0, 1)$ we can ensure that $I$ satisfies the Palais-Smale condition. 

\medskip

Thus, for each $k\in\N$, let $X_{k}$ be a $k$-dimensional subspace of $H_{0}^{1}(\Omega)$. Let also $S_{1}:=\{v\in H_{0}^{1}(\Omega):\|v\|=1\}$. Since $S_{1}\cap X_{k}$ is compact, by Lemma \ref{tec}, there exist $s_{k}>0$ and $\beta_{k}>0$ such that
$$
S_{s_{k}}\cap X_{k}=(s_{k}S_{1})\cap X_{k}\subset \mathcal{A}_{k}:=\{v\in X_{k}: |[|v|>1]|\geq\beta_{k}\}. 
$$

Finally, by Proposition \ref{pp2}$(v)$, we get
\begin{equation}\label{count}
|f(s)|\geq f(1)\sqrt{|s|}, \ \forall \ |s|> 1.
\end{equation}
Therefore, by \eqref{count}
\begin{eqnarray*}
I(s_{k}v)&\leq& \frac{1}{2}s_{k}^{2}-\frac{\lambda f(1)^{(q+1)/2}}{q+1}\int_{[|s_{k}v|>1]}|s_{k}v|^{(q+1)/2}dx\\
&\leq & \frac{1}{2}s_{k}^{2}-\frac{\lambda f(1)^{(q+1)/2}}{q+1}|[|s_{k}v|>1]|\\
&\leq &  \frac{1}{2}s_{k}^{2}-\frac{\lambda f(1)^{(q+1)/2}}{q+1}\beta_{k}, \ \forall \ v\in S_{1}\cap X_{k}.\\
\end{eqnarray*}
Showing that
$$
\sup_{w\in S_{s_{k}}\cap X_{k}}I(w)<0,
$$
whenever $\lambda> \lambda_{k}:=(q+1)s_{k}^{2}/2f(1)^{(q+1)/2}\beta_{k}$. By the classical Clark Theorem in \cite{Cla}, it follows that \eqref{P2'} has at least $k$ pairs of nontrivial solution with negative energy.
$\square$

%
%


\subsection{Case $q=3$}.\\

{\bf Proof of Theorem \ref{teor4}:}

\medskip

It follows from items $(v)$ and $(viii)$ of Lemma \ref{main} that
$$
\frac{g(s)}{s}<\frac{4}{\alpha^{2}}, \ \forall \ s>0.
$$
Thus, if $v_{\lambda}$ is a positive solution of $\eqref{P2'}$, we obtain
$$
0=\lambda_{1}\left(-\Delta-\lambda\frac{g(v_{\lambda})}{v_{\lambda}}\right)>\lambda_{1}\left(-\Delta-\lambda\frac{4}{\alpha^{2}}\right)=\lambda_{1}-\lambda\frac{4}{\alpha^{2}}.
$$
Showing that, if there exists a positive solution of \eqref{P2}, then $\lambda>(\alpha^{2}/4)\lambda_{1}$. 

In this Theorem and in the next one, we are going to apply now the bifurcation method. For this, let $e$ denote the unique positive solution of (\ref{prob}) in $\Omega$
and let $E$ be the Banach space consisting of all $u \in C(\overline{\Omega})$ for which there exists $\gamma = \gamma(u)>0$ such that
$-\gamma e < u < \gamma e $
endowed with the norm
$$\|u\|_E:= \mbox{inf}\{\gamma>0;~ -\gamma e< u < \gamma e\}$$ and the natural point-wise order. Then, $E$ is an ordered Banach space whose positive cone, say $P$, is normal and has nonempty interior. Moreover, $E\hookrightarrow C(\overline\Omega)$.

On the other hand, by Lemma \ref{main}$(v)$ we have that
$$
\lim_{s\to +\infty}\frac{g(s)}{s}=\frac{\alpha^{2}}{4}.
$$ 
Hence, we can apply Theorem 7.1.3 of \cite{julian}, see also Theorem A of \cite{ah}, and conclude that  
from $\lambda=(\alpha^{2}/4)\lambda_{1}$ emanates from infinity an unbounded continuum $\mathcal{C}\subset\R\times E$ of positive solutions and $\lambda=(\alpha^{2}/4)\lambda_{1}$ is the unique bifurcation point from infinity. We are going to show that $\mathcal{C}\cap(\R\times\{0\})=\emptyset$. In fact, otherwise, there exists a couple of sequences $\lambda_{n}\to\lambda_{\ast}\in (0,\infty)$ and $|v_{n}|_{\infty}\to0$ where $v_{n}$ is a positive solution of \eqref{P2'} with $\lambda=\lambda_{n}$. By Lemma \ref{main}$(iii)$, for all $\varepsilon>0$ there exists $n_{0}\in\N$ such that
$$
\frac{g(v_{n})}{v_{n}}<\varepsilon, \ \forall \ n\geq n_{0}.
$$
Consequently,
$$
0=\lambda_{1}\left(-\Delta-\lambda_{n}\frac{g(v_{n})}{v_{n}}\right)>\lambda_{1}-\lambda_{n}\varepsilon.
$$
Thus, $\lambda_{\ast}\varepsilon\geq\lambda_{1}$ for all positive $\varepsilon$, which leads us to a contradiction. Therefore, $Proj_{\R}(\mathcal{C})=((\alpha^{2}/4)\lambda_{1}, \infty)$, where $Proj_{\R}(\mathcal{C})$ denotes the projection of $\mathcal{C}$ on $\R$.
$\square$

\medskip

\subsection{Case $3<q<22^{\ast}-1$}.\\

{\bf Proof of Theorem \ref{teor5}(a):}

\medskip

Fix $\lambda>0$. We first prove that the problem
\begin{equation}\label{isso}
\left \{ \begin{array}{ll}
-\Delta w=\mu w+\lambda g(w) & \mbox{in $\Omega$,}\\
w>0 & \mbox{in $\Omega$,}\\
w=0 & \mbox{on $\partial\Omega$.}
\end{array}\right.
\end{equation}
has a positive solution if, and only if, $\mu\in(-\infty, \lambda_{1})$. Indeed, if $v$ is a positive solution of the previous problem, then
$$
\mu=\lambda_{1}\left(-\Delta-\lambda\frac{g(v)}{v}\right)<\lambda_{1}.
$$

On the other hand, by using the Lemma \ref{main}$(iii)$ we get tha
$$
\lim_{s\to 0}\frac{\l g(s)}{s}=0,
$$
and from Theorem 7.1.3 in \cite{julian} we conclude that from $\mu=\lambda_{1}$ emanates, from $w=0$, an unbounded continuum $\mathcal{C}\subset\R\times E$ of solutions of \eqref{isso}. Moreover, since $q\in (2, 22^{\ast}-1)$, then 
\begin{equation}\label{nois}
1<\frac{q-1}{2}<\frac{N+2}{N-2}.
\end{equation}
and, by Proposition \ref{pp2}$(vi)$, we get
$$
\lim_{s\to\infty}\frac{g(s)}{s^{(q-1)/2}}=\lim_{s\to\infty}\left(\frac{f(s)}{\sqrt{s}}\right)^{q-1}\times\lim_{s\to\infty}\sqrt{\frac{1}{\vartheta(f(s))/f(s)^{2}}}=\left(\frac{8}{\alpha^{2}}\right)^{(q-1)/4}\frac{\sqrt{2}}{\alpha},
$$
it follows from \cite{GS} that, for $\mu\in K$, $K\subset\R$ compact, $\{|v_{\mu}|_{\infty}\}$ is bounded, and by elliptic regularity, $\{v_\m\}$ is also bounded in $E$. Therefore, $Proj_{\R}(\mathcal{C})=(-\infty, \lambda_{1})$. Showing the claimed. Consequently, if $\mu=0$, there exists a positive solution of \eqref{P2} for all $\lambda>0$ and the result follows.

Now, we prove that $\lim_{\lambda\to 0}|u_\lambda |_\infty=\infty$. Assume that for a sequence $|u_{\lambda_n}|_\infty\leq C$. Then, by elliptic regularity, we conclude that $u_{\lambda_n}\to u_0\geq 0$ in $C^2(\overline\Omega)$ with $u_0$ a non-negative solution of  \eqref{P2} for $\lambda=0$. If $u_0$ is non-trivial, we arrive at a contradiction. If $u_0\equiv 0$, then we consider
$$
v_{\lambda_n}=\frac{u_{\lambda_n}}{|u_{\lambda_n}|_\infty}.
$$
Hence, $|v_{\lambda_n}|_\infty=1$, and by a similar argument to the above one, we conclude that $v_{\lambda_n}\to v_0>0$ in $C^2(\overline\Omega)$ and $v_0$ solution of \eqref{P2} for $\lambda=0$, a contradiction.
$\square$

\medskip

{\bf Proof of Theorem \ref{teor5}(b):}

\medskip

The proof is based in the symmetric mountain pass lemma in \cite{Rab}. For some $w\in H_{0}^{1}(\Omega)$, with $\|w\|=1$, we can split $H_{0}^{1}(\Omega)$ in the following way $H_{0}^{1}(\Omega)=X\oplus Span\{w\}$, where $X$ is the orthogonal complement of $w$. It follows from Proposition \ref{pp2}$(iv)$ and Sobolev embeddings, that
$$
I(sv)=\frac{1}{2}s^{2}-\frac{\lambda}{q+1}\int_{\Omega}|f(sv)|^{q+1}dx\geq  \frac{1}{2}s^{2}-\lambda C|s|^{q+1},
$$
for all $v\in X$ with $\|v\|=1$ and some positive $C$. Since $q\in (3, 22^{\ast}-1)$, there exist positive constants $\rho$ and $\alpha$ such that 
$$
I(\rho v)\geq \alpha, \ \forall \ v\in X \ \mbox{with $\|v\|=1$}.
$$

Now we are going to prove that for each $k$-dimensional subspace $\varphi_{k}$ of $H_{0}^{1}(\Omega)$, with $k>1$, there exist $\gamma_{k}>0$ and $r_{k}>0$ such that
\begin{equation}\label{neg}
I(v)\leq 0, \ \forall \ v\in \varphi_{k}\backslash B_{r_{k}}(0).
\end{equation}
For this, it is sufficient to note that, by Lemma \ref{tec}, there exist $\beta_{k}>0$ and $r_{k}>0$ such that
$$
\varphi_{k}\backslash B_{r_{k}}(0)\subset \ \mathcal{A}_{k}:=\mathcal{A}_{\beta_k}.
$$
Thus, by \eqref{count}
$$
I(v)\leq \frac{1}{2}\|v\|^{2}-\frac{\lambda}{q+1}\int_{[v>1]}|f(v)|^{q+1}dx\leq \frac{1}{2}\|v\|^{2}-\frac{\lambda f(1)^{q+1}}{q+1}\int_{[v>1]}|v|^{(q+1)/2}dx, 
$$
for all $v\in \varphi_{k}\backslash B_{r_{k}}(0)$. Hence,
\begin{eqnarray*}
I(v)&\leq& \frac{1}{2}\|v\|^{2}-\frac{\lambda f(1)^{q+1}}{q+1}\int_{\Omega}|v|^{(q+1)/2}dx+\frac{\lambda f(1)^{q+1}}{q+1}|\Omega|\\
&\leq& \frac{1}{2}\|v\|^{2}-\lambda C_{k}\|v\|^{(q+1)/2}+\lambda C,
\end{eqnarray*}
for all $v\in \varphi_{k}\backslash B_{r_{k}}(0)$. Since $q\in (3, 22^{\ast}-1)$, we can choose $r_{k}$ large enough in order to ensure that
$$
I(v)\leq 0, \ \forall \ v\in \varphi_{k}\backslash B_{r_{k}}(0).
$$
Finally, to show that $I$ satisfies the $(PS)_{c}$ condition, let $\{v_{n}\}\subset H_{0}^{1}(\Omega)$ such that
$$
I(v_{n})\to c \ \mbox{and} \ \|I'(v_{n})\|\to 0.
$$
Choosing $\varphi_{n}=f(v_{n})/f'(v_{n})$, we conclude from Proposition \ref{pp2}$(v)$ that
$$
|\varphi_{n}|\leq 2 |v_{n}| \ \mbox{and} \ |\nabla\varphi_{n}|=\left(1+\frac{\vartheta'(f(v_{n}))f(v_{n})}{2\vartheta(f(v_{n}))}\right)|\nabla v_{n}|. 
$$ 
Consequently, by $(\vartheta_{2})$,
\begin{equation}\label{eita!}
|\varphi_{n}|_{2}\leq 2|v_{n}|_{2} \ \mbox{and} \ \|\varphi_{n}\|\leq 2\|v_{n}\|,
\end{equation}
showing that $\varphi_{n}\in H_{0}^{1}(\Omega)$. Let us see now that $\{v_{n}\}$ is bounded. In fact, by \eqref{eita!}
$$
C_{1}+C_{2}\|v_{n}\|\geq I(v_{n})-\frac{1}{q+1}I'(v_{n})\varphi_{n}= \left[\frac{1}{2}-\frac{1}{q+1}\left(1+\frac{\vartheta'(f(v_{n}))f(v_{n})}{2\vartheta(f(v_{n}))}\right)\right]\|v_{n}\|^{2}.
$$
By using $(\vartheta_{2})$ again, we have
$$
C_{1}+C_{2}\|v_{n}\|\geq \left[\frac{1}{2}-\frac{2}{q+1}\right]\|v_{n}\|^{2}= \frac{(q-3)}{2(q+1)}\|v_{n}\|^{2}, \ \forall \ n\in\N.
$$
Therefore $\{\|v_{n}\|\}$ is bounded. To prove that $\{u_{n}\}$ has a convergent (in $H_{0}^{1}(\Omega)$) subsequence, it is sufficient to argue as in the proof of Theorem \ref{teor2}. The result follows now from the symmetric mountain pass lemma in \cite{Rab}.
$\square$

\subsection{Case $q\geq 22^{\ast}-1$}.\\

{\bf Proof of Theorem \ref{teor6}:}

\medskip

Define the function
$$
z(s)=\frac{(N-2)}{2}g(s)s-NG(s), \ \forall s\in\R.
$$
Observe that $z(0)=0$ and
$$
z'(s)=\frac{(N-2)}{2}g'(s)s-\frac{(N+2)}{2}g(s).
$$
Thus $z'(s)\geq 0$ ($s> 0$) if, and only if,
$$
\frac{g(s)}{g'(s)s}\leq \frac{N-2}{N+2}.
$$
Since,
\begin{eqnarray*}
\frac{g(s)}{g'(s)s}=\frac{2f(s)\vartheta(f(s))^{3/2}}{\left(2q\vartheta(f(s))-\vartheta'(f(s))f(s)\right)s}, \ \forall \ s>0,
\end{eqnarray*}
it follows from $(\vartheta_{1})-(\vartheta_{2})$ that
$$
\frac{g(s)}{g'(s)s}<\frac{2f(s)\vartheta(f(s))^{1/2}}{\left(2q-1\right)s}=\frac{2f(s)}{\left(2q-1\right)f'(s)s}.
$$
Finally, by Proposition \ref{pp2}$(v)$, we get
$$
\frac{g(s)}{g'(s)s}<\frac{2}{q-1}\leq \frac{N+2}{N-2},
$$
where the last inequality follows from $q\in [22^{\ast}-1, \infty)$. So, if there is a positive solution $u$ of \eqref{P2}, by Pohozaev inequality, we conclude that
$$
0\leq \int_{\Omega}\left(\frac{(N-2)}{2}g(u)u-NG(u)\right)dx< 0.
$$
A clear contradiction.
$\square$


\end{document}